\documentclass[%11pt,%BCOR=1cm,
a4paper
]{amsart}
\usepackage[utf8]{inputenc}
\usepackage[margin=1.2in]{geometry}
\usepackage[
doi=false,
backend=biber,
style=alphabetic,
]{biblatex}
\usepackage{mathtools}
\usepackage[dvipsnames]{xcolor}
\usepackage{amssymb, amsthm, mdframed, amsmath, pifont} %f{\"u}r Symbole der reellen Zahlen
\usepackage{multicol}
\usepackage{amsfonts}
\usepackage{enumerate}
\usepackage{pdfpages}
\usepackage{graphicx}
\usepackage{psfrag}
\usepackage{verbatim}
\usepackage{pstricks}

\usepackage{thm-restate}
\usepackage{mathrsfs}
\usepackage[all,knot]{xy}
\usepackage{tikz}
\usepackage{caption}
\usepackage[lofdepth,lotdepth]{subfig}
\usepackage{thm-restate}

\usepackage{thmtools}
\usepackage{thm-restate}

\usepackage[colorlinks=true, citecolor=teal,linkcolor=violet]{hyperref}
\usepackage{wrapfig}
\usepackage[linguistics]{forest}

\usepackage{cleveref}
\usepackage{todonotes}

\usepackage[section]{algorithm}
\usepackage{algpseudocode}

\usepackage{aliascnt}

\algtext*{EndWhile}% Remove "end while" text
\algtext*{EndIf}% Remove "end if" text

\declaretheorem[name=Theorem]{thm}

\newtheorem{lem}[thm]{Lemma}
\newtheorem{cor}[thm]{Corollary}

\newtheorem{prop}[thm]{Proposition}
\newtheorem{lemma}[thm]{Lemma}

\theoremstyle{definition}
\newtheorem{defn}[thm]{Definition}
\newtheorem{question}[thm]{Question}
\newtheorem{example}[thm]{Example}
\newtheorem{definition}[thm]{Definition}
\newtheorem{prob}[thm]{Problem}
\newtheorem*{intdef}{Definition}
\newtheorem*{claim*}{Claim}
\newtheorem{remark}[thm]{Remark}

\newtheorem{claim}[thm]{Claim}
\newtheorem{obs}[thm]{Observation}

\newtheorem{numalg}[thm]{Algorithm}
\newtheorem*{alg*}{Algorithm}

\newenvironment{prettyalg}[2]
    {
    \vspace{.5cm}
    \hrule
    \vspace{-1.3mm}
    \begin{numalg}\label{#2} #1
    \end{numalg}
    \vspace{-1.3mm}
    \hrule
    \vspace{1mm}
    }
{\vspace{1mm}     \hrule
 \vspace{.5cm}}

\def\s{{\sigma}}

\def\epsilon{\varepsilon}
\def\Mod{\mbox{Mod}}
\def\N{\mathbb{N}}

\def\Z{\mathbb{Z}}

\def\aut{\rm{Aut}}
\def\inn{\rm{Inn}}
\newcommand{\rt}{\mathbf{return}\,}

\begin{document}
\title{Searching for non-order-preserving braids algorithmically }
\author{Jonathan Johnson, Nancy Scherich, and Hannah Turner} 

\address{Sam Houston State University}
\email{jcj055@shsu.edu}

\address{Elon University}
\address{2320 Campus Box
Elon, NC 27244 USA}
\email{nscherich@elon.edu}

\address{Stockton University}
\email{hannah.turner@stockton.edu}


\begin{abstract}
An $n$-strand braid is order-preserving if its action on the free group $F_n$ preserves some bi-order of $F_n$. A braid $\beta$ is order-preserving if and only if the link $L$ obtained as the union of the closure of $\beta$ and its axis has bi-orderable complement. We describe and implement an algorithm which, given a non-order-preserving braid $\beta$, confirms this property and returns a proof that $\beta$ is indeed not order-preserving.
Guided by the algorithm, we prove that the infinite family of simple 3-braids $\sigma_1\sigma_2^{2m+1}$ are not order-preserving for any integer $m$.
\end{abstract}

\maketitle
\section{Introduction}
A group $G$ is called \emph{bi-orderable} if there is a strict total ordering  on $G$ that is invariant under both left and right multiplication.
Free groups are bi-orderable, and in fact have uncountably many distinct bi-orders.
Given a braid $\beta$ in the $n$-strand braid group $B_n$, there is a specific action of $\beta\in B_n$  on the free group $F_n$ of rank $n$ coming from the induced action of $\beta$ on the fundamental group of the $n$-punctured disk; see Section \ref{sect:braidsandorders}. We aim to classify braids via properties of this action with respect to bi-orders of the free group. In particular, if the action of $\beta$ preserves at least one bi-order of $F_n$, then we say that $\beta$ is \emph{order-preserving}.

%or that $\beta$ preserves a bi-order of $F_n$.

\begin{question} \label{ques:OP}
Which braids are order-preserving?
%preserve a bi-order of the free group of rank $n$?
\end{question}

Kin and Rolfsen pioneered the study of order-preserving braids and have answered Question \ref{ques:OP} for several families of braids \cite{KR-BraidsOrderings}. 
We present Algorithm \ref{The_alg} which will certify that a braid is \emph{not} order-preserving in finite time.
The goal of the algorithm is to generate all possible bi-orders of the free group that could be preserved by a given input braid.
To turn this infinite problem into a finite problem,
Algorithm \ref{The_alg} uses an equivalent definition of bi-orders using positive cones, and then truncations of these cones to finite sets called $k$-precones, see Section \ref{sect:braidsandorders}. 
The algortihm searches for all $k$-precones preserved by the chosen braid for a fixed $k$. The following proposition allows us to conclude that when this search fails, the braid is not order-preserving. 

\begin{restatable}{prop}{whyalgworks}
    An $n$-strand braid $\beta$ is order-preserving if and only if $\beta$ preserves a $k$-precone of the free group $F_n$ for every positive integer $k$. 
    \label{prop:whyalgworks}
\end{restatable}

\begin{restatable}{numalg}{} (Summary)
The algorithm takes as input a braid $\beta$ on $n$ strands. For each positive integer $k$ starting with $k=1$, the algorithm builds all $k$-precones of the free group of rank $n$ that are preserved by $\beta$.
If no $k$-precones exist, the algorithm terminates and returns a proof that $\beta$ is not order-preserving.
If there exists at least one $k$-precone preserved by $\beta$, $k$ is increased by one and the process repeats.
\label{The_alg}
\end{restatable}

\begin{restatable}{thm}{alg_terminates}
Given a braid which is not order-preserving, Algorithm \ref{The_alg} will terminate in finite time.
\label{thm:alg_terminates}
\end{restatable}

Algorithm \ref{The_alg} is implemented in Python and available for use at \cite{JST23}. Our algorithm and its implementation are inspired by an algorithm of Calegari and Dunfield which certifies non-left-orderability of a finitely presented group in finite time; if the group is left-orderable, their algorithm will not halt \cite{CD03, ND20}.
Similarly, our algorithm will not terminate if applied to a braid that is order-preserving.
Because our algorithm is applied to braids, we implement braid-specific improvements to increase the efficiency of Algorithm \ref{The_alg}. For instance, the following theorem allows us to restrict to building positive cones for the subgroup of words in $F_n$ whose exponent sum is zero.

\begin{restatable}{thm}{allgenspos}
   An $n$-strand braid $\beta$ is order-preserving if and only if $\beta$ preserves a positive cone $P$ where any word in $F_n$ with positive exponent sum is in $P$.  
\label{thm:allgenspos}
\end{restatable}

Even with these braid-specific advantages, it is natural to question how long it will take Algorithm \ref{The_alg} to terminate.
For a fixed $k$, one can ask the computational complexity of the implementation of Algorithm \ref{The_alg}. In terms of answering Question \ref{ques:OP} we are also interested in knowing how large  $k$ needs to be in order to find an obstruction to the existence of a preserved $k$-precone.

\begin{question}\label{ques:open}
For a non-order-preserving braid $\beta$, how large does $k$ need to be for  Algorithm \ref{The_alg} to determine that $\beta$ does not preserve a $k$-precone? Is there a relationship between the length of the braid word for $\beta$ and the minimal $k$  for  Algorithm \ref{The_alg} to find a contradiction?
\end{question}

    In Section \ref{sec:alg}, we discuss some complexity measurements of parts of Algorithm \ref{The_alg}, but in general, Question \ref{ques:open} is still open.
In fact, if one could answer  Question \ref{ques:open}, then Algorithm \ref{The_alg} would become an algorithm to detect order-preserving braids, which is much stronger than the current claim of this paper.
Regardless of the theoretical measures of efficiency,  the implemented Algorithm \ref{The_alg} found a an order-preserving obstruction for the braid $\sigma_1\sigma_2^{-3}$ with $k=4$ in a matter of seconds.
The implemented algorithm not only found the contradiction, but outputted a proof that $\sigma_1\sigma_2^{-3}$ was not order-preserving. From this proof, we were able to generalize the argument to show a new infinite family of braids are not order-preserving. 

\begin{restatable}{thm}{braidfam}
The braids $\sigma_1\sigma_2^{2m+1}$ are not order-preserving for any integer $m$.
\label{thm:newbraidfamily}
\end{restatable}

This family is among the simplest class of braids whose order-preserving properties remained unknown.
The braid group is isomorphic to the mapping class group of the punctured disk $\Mod(D_n)$. With this in mind, braids can be classified by their Nielsen-Thurston type as either periodic, pseudo-Anosov, or reducible. Kin and Rolfsen classified which 2-braids and which periodic braids are order-preserving \cite{KR-BraidsOrderings}. A natural next class to consider is pseudo-Anosov braids 3-braids.

Every psuedo-Anosov 3-braid is conjugate to  $h^d\sigma_1\sigma_2^{-a_1}\cdots \sigma_1\sigma_2^{-a_n}$ with $a_i\geq 0$ with at least one $a_i\neq 0$ where $h=(\sigma_1\sigma_2)^3$ is the full twist \cite{Mur-3braids}. Considering this classification, the simplest family of 3-braids is $\sigma_1\sigma_2^{-a}.$ We note that Kin and Rolsen showed that when $a=1$ the braid is not order-preserving; Theorem \ref{thm:newbraidfamily} extends this result to an infinite family when $a$ is odd. Recent work of Cai, Clay, and Rolfsen shows that these braids \emph{are} order-preserving when $a\equiv 2$ mod $4$ \cite{CCR-Ordered-bases}. One of our goals in creating and implementing the above algorithm is to increase examples of braids known to be not order-preserving, especially among pseudo-Anosov 3-braids. Theorem \ref{thm:newbraidfamily} is a concrete step towards this goal.

Theorem \ref{thm:newbraidfamily} is written as an algebraic statement about the braid group; there is however, a topological interpretation of our result. 
The study of order-preserving braids fits into a much larger schema of studying the bi-orderability of link complements \cite{CCR-Ordered-bases,CDN16,Ito17,JJ1,JJ2,PerRolf03,Yam17}. 
A link, $L$, is said to be \emph{bi-orderable} if $\pi_1(S^3-L)$ is bi-orderable.

\begin{prob}
\label{ques:bio-links}
Classify bi-orderable links in $S^3$.
\end{prob}

A \emph{braided link} is the closure $\widehat{\beta}$ of an $n$-strand braid $\beta$ together with the braid axis, as pictured in Figure \ref{pic:closedbraid}.
Utilizing the structure of the braided link complement, Kin-Rolfsen show that Problem \ref{ques:bio-links} is equivalent to Question \ref{ques:OP} for braided links \cite{KR-BraidsOrderings}. Algorithm \ref{The_alg} thus finds bi-orderability obstructions for braided links. Furthermore, Theorem \ref{thm:newbraidfamily} determines that the links obtained from taking the closure of $\beta=\sigma_1\sigma_2^{2m+1}$ together with their braid axes, are not bi-orderable for any integer $k.$

\subsection{Organization of the paper.} In Section \ref{sect:braidsandorders} we define the explicit action of a braid on the free group that is used to define order-preservingness of a braid, and organize some background information about braids and bi-orders of groups. Finding a bi-order of the free group preserved by a braid can be reduced to finding a certain order on zero-exponent sum words in the free group, as we show in Section \ref{sec:0-sum}. We describe our algorithm and its implementation in Section \ref{sec:alg}. In Section \ref{sec:Family}, we prove Theorem \ref{thm:newbraidfamily}.
\subsection{Acknowledgments}
This material is based upon work supported by then NSF grant DMS-1929284 while the three authors were in residence at the Institute for Computational and Experimental Research in Mathematics in Providence, RI, during the `Braids' program. We continued the collaboration through the Collaborate@ICERM program. We are very grateful to ICERM for making the collaboration possible and supporting our work. Part of this research was conducted using computational
resources and services at the Center for Computation and Visualization, Brown University. The first author was supported in part by NSF grant DMS-2213213. The second author was partially supported by an AMS-Simons travel grant and NSF grant DMS-2532699. The third author was also partially supported by NSF grant DMS-1745583.

\section{Braids and orders} \label{sect:braidsandorders}

Throughout this section we use Artin's presentation of the $n$-strand braid group $B_n;$ see Figure \ref{pic:Artingen} for our convention for the generators $\sigma_i$. The braid group $B_n$ embeds as a subgroup of the automorphism group,  $\text{Aut}(F_n)$, of the free group, $F_n$, via the following action, where $F_n=\langle x_1,\cdots ,x_n\rangle$.
 
 \begin{equation} \label{eq:artin}
 \sigma_i\mapsto \begin{cases}
 x_i\mapsto x_{i+1} \\
  x_{i+1}\mapsto x_{i+1}^{-1}x_{i}x_{i+1} \\
  x_j\mapsto x_j 
  \end{cases}
 \end{equation}

 \begin{figure}
\centering
\subfloat[Subfigure 1 list of figures text][]{\includegraphics[width=0.25\textwidth]{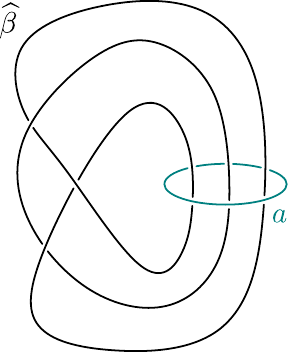}
\label{pic:closedbraid}
}
\quad\qquad
\subfloat[Subfigure 2 list of figures text][]{
\includegraphics[width=0.3\textwidth]{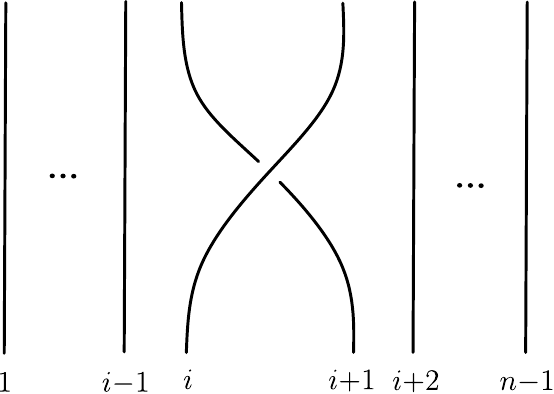}
\label{pic:Artingen}}

\caption{(A) The closure $\widehat{\beta}$ of %the 3-braid $\beta=\sigma_1\sigma_2^{-1}\sigma_1$
 of a 3-braid $\beta$ together with its axis $a$ forms a braided link. (B) The Artin generator $\sigma_i$. }
\label{pic:introtobraids}
\end{figure}

 This action comes from the identification of $B_n$ with the mapping class group $\Mod(D_n)$ of the $n$-punctured disk. Thus the braid group acts on $\pi_1(D_n)\cong F_n$. We read braid words from \emph{right to left} so that braids act on elements of $F_n$ on the left. This also means that we read the action of $\sigma_i$ on the punctured disk by tracing the paths of the strands as we flow up the braid so that the $i^{th}$ puncture passes in front of the $(i+1)^{th}$ puncture; see Figure \ref{fig:beta-action} for our conventions.
% \todo[color = cyan]{(JCJ) A picture showing our conventions would be nice.}
% Our convention for the generators of $\pi_1(D_n)$ is that the basepoint is chosen along the bottom of the disk and the $i^{th}$ generator loops once in a clockwise direction around the $i^{th}$ puncture. 
We note that the action of $\beta$ on $F_n$ with our conventions is the inverse automorphism which Kin and Rolfsen consider for the same $\beta.$

\begin{figure}
    \centering
    \includegraphics[width=0.9\linewidth]{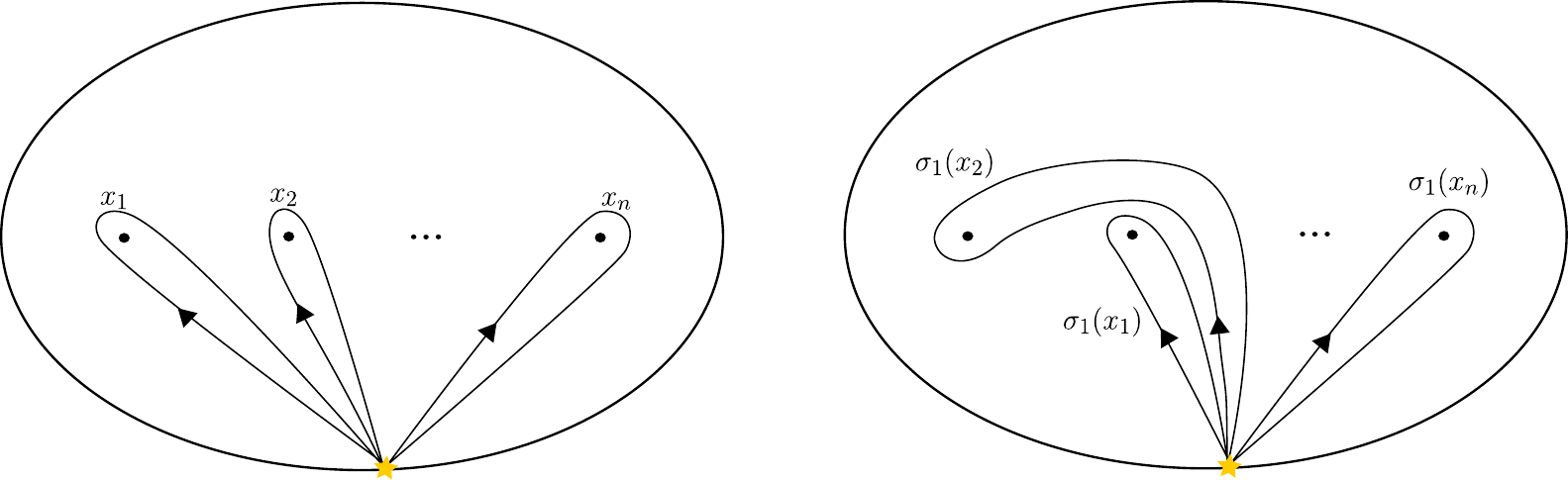}
    \caption{Our conventions for the induced action of the $n$-strand braid $\sigma_1$ on the generators of $\pi_1(D_n)$.}
    \label{fig:beta-action}
\end{figure}

\begin{definition}
\label{def:poscone}
    A subset $P\subset F_n$ is a \emph{(conjugate invariant) positive cone} if
    \begin{enumerate}
        \item $P\cdot P\subset P$,
        \item $F_n=P\sqcup P^{-1}\sqcup\{1\}$, and
        \item $gPg^{-1}=P$ for all $g\in F$.
    \end{enumerate}
\end{definition}

A positive cone $P$ determines a bi-order in the following way: say that $f<g$ if and only if $f^{-1}g\in P$. This definition automatically guarantees the order will be left-invariant. Condidition (1) assures transitivity, and (2) gives totality and strictness. Condition (3) is equivalent to $<$ being right-invariant. Thus defining $<$ in this way from $P$ gives a bi-order of $F_n$. Conversely, given a bi-order $<$ on $F_n$, the set of elements greater than the identity determines a positive cone of $F_n$.

\begin{definition}
    An $n$-strand braid $\beta$ is called \emph{order-preserving} if there exists a positive cone $P$ of $F_n$ preserved by $\beta$. That is $\beta(P)=P$, set-wise.
\end{definition}

\begin{remark}\label{rem:favorite_elt}
Given a positive cone $P$, the set $P^{-1}$ is also a positive cone (for the opposite order obtained by flipping all inequalities). Furthermore, a braid preserves $P$ if and only if it preserves $P^{-1}$. Because of this, when $\beta$ preserves any positive cone $P$ of $F_n$, there is no loss of generality in assuming that your favorite non-trivial element of $F_n$ is in $P$. 
\end{remark}

For each braid $\beta$ we obtain a link in $S^3$ by taking the union of the closure $\hat{\beta}$ of $\beta$ with the braid axis $a$; see Figure \ref{pic:closedbraid}.

\begin{prop}[{\cite[Proposition 4.1]{KR-BraidsOrderings}}] The braided link $\hat{\beta}\cup a$ is bi-orderable if and only if the action $\beta$ on $F_n$ preserves some bi-order on $F_n$
    
\end{prop}

\subsection{Precones}\label{sec:precones}
Let $k\geq 1$, and
let $W_k$ be the set of reduced words in $x_1,\ldots,x_n$ of length less than or equal to $k$. A $k$-precone of $F_n$ is the analog of a positive cone of $F_n$ restricted to words of length $k$, as made precise in the following definition.

\begin{definition}\label{def:kprecone}
    A subset $P$ is a \emph{$k$-precone} of $F_n$ if
    \begin{enumerate}
        \item $(P\cdot P)\cap W_k\subset P$,
        \item $W_k=P\sqcup P^{-1}\sqcup\{1\}$, and
        \item $(gPg^{-1})\cap W_k\subset P$ for all $g\in W_k$.
    \end{enumerate}
\end{definition}

Notice that a $k$-precone is not necessarily closed under multiplication (or conjugation) since many products of elements in the precone may be too long. Given a subset $S$ of $F_n$ there is an action of the braid on this subset which we denote by $\beta(S)$.

\begin{definition} \label{def:preconepreserve}
    
    Given $k\in\N$, a $k$-precone $P_k$ is \emph{preserved by} an automorphsim $\varphi$ if the intersection\\$\varphi(P_k)\cap W_k\subset P_k$. 
\end{definition}

We point out that if $P$ is a positive cone of $F_n$ preserved by $\beta$, then $P\cap W_k$ for any positive integer $k$ is a $k$-precone preserved by $\beta$ by checking the definitions. There can be many different cones that have the same $k$-precone for a given $k$.
The following proposition also asserts the converse statement.

\whyalgworks*

Proposition \ref{prop:whyalgworks} is the crucial result used in our algorithm.
To show a braid $\beta$ does not preserve a positive cone of $F_n$, it suffices to show that for some $k$, $\beta$ does not preserve any $k$-precones of $F_n$.
For any fixed $k$, there are a finite number of $k$-precones of $F_n$, each with finite cardinality. 
So Proposition \ref{prop:whyalgworks} reduces the infinite problem to a finite problem -- assuming the braid does not preserve any order.

The proof of Proposition \ref{prop:whyalgworks} uses a well-known style of argument which is described in the proof of a similar statement in Section 1.6 of Clay and Rolfsen's book \emph{Ordered Groups and Topology} \cite{CR16}. For the convenience of the reader, we've included a proof here.

Towards this goal, we prove some necessary lemmas. First, we show that the union of a family of nested precones is a positive cone preserved by the braid $\beta$.

\begin{lem} \label{lem:preconetocone}
    Suppose that for each positive integer $k$, we have a $k$-precone $P_k$ of $F_n$ preserved by $\beta$. If $P_k\subset P_l$ for all $k\leq l$, then $P=\bigcup P_k$ is a positive cone of $F$ preserved by $\beta$.
\end{lem}

\begin{proof}
    To show that $P$ is positive cone preserved by $\beta$ we need to check the three conditions of Definition \ref{def:poscone}, and finally to check that $\beta$ preserves $P$.  
    
    To do these checks, it will be convenient first to show that $P\cap W_k=P_k$. For each $k$ we certainly have that $P_k\subset P\cap W_k$.    
    Now suppose $x\in P\cap W_k$ for some $k$.
    Since $x\in P$, $x\in P_l$ for some $l$.
    If $l\leq k$ then $x\in P_l\subset P_k$.
    Suppose $l>k$.
    Since $x\in W_k$, either $x\in P_k$, $x\in P_k^{-1}$, or $x=1$.
    Since $x\in P_l$, $x\neq 1$.
    Also, since $P_k^{-1}\subset P_l^{-1}$ and $x \in P_l$,  $x$ cannot be in $P_l^{-1}$, nor can $x$ be in $P_k^{-1}$.
    Thus, we must have $x\in P_k$ so $P\cap W_k = P_k$.

    \textbf{Condition (1):}
    Suppose $a,b\in P$.
    For some large enough $k$, we must have $a, b,$ and $ ab\in W_k$.
    Since $P_k=P\cap W_k$, $a,b\in P_k$.
    Since $P_k$ is a precone, $ab\in P_k\subset P$ as desired.

    \textbf{Condition (2):}
    Since $P=\cup P_k$ we also have that $P^{-1}=\cup P_k^{-1}$; we claim that $P\sqcup P^{-1}\sqcup \{1\}=F_n$. Suppose $g\in F_n$ so $g\in W_k$ for some $k$.
    Thus, either we have that $g\in P_k\subset P$, or $g\in P_k^{-1}\subset P^{-1}$, or $g=1$. In any case, $g$ is in $P\sqcup P^{-1}\sqcup \{1\}$ and hence we have that $F_n=P\sqcup P^{-1}\sqcup\{1\}$.

    \textbf{Condition (3):}
    Suppose $g\in F_n$ and $x\in P$.
    For some $k$,we have that $x, gxg^{-1}\in W_k$.
    Since $P_k=P\cap W_k$ and is a precone, we also have that $ gxg^{-1}\in P_k\subset P$.

    \textbf{Preserved by $\beta$:}
    Suppose $x\in P$.
    For some $k$, we have that $x, \beta(x)\in W_k$. Since $P_k=P\cap W_k$ and is a $k$-precone preserved by $\beta$, we have that $\beta(x)\in P_k\subset P$.
\end{proof}

To complete the proof of Proposition \ref{prop:whyalgworks} we need to show that we have a \emph{nested} set of precones. Before showing this, we need the following a point-set topology result.

\begin{lem}[{\cite[Theorem 26.9]{Munkres2000}}]\label{lem:nestednonempty} 
    Suppose $X$ is a compact space with a countable family $\mathcal{C}$ of closed nested subsets.
    If each $C\in\mathcal{C}$ is nonempty then the intersection of all $C\in\mathcal{C}$ is also nonempty.
\end{lem}

\begin{lem}
\label{lem:preconesnested}
    Suppose that $\beta$ preserves a $k$-precone for each positive integer $k$. Then $\beta$ preserves a set of $k$-precones for each positive integer $k$ which are nested.
\end{lem}

\begin{proof}
  Consider the powerset of $F_n$, denoted $2^F$.
    Each $A\in 2^F$ is identified with an indicator function $f_A:F_n\to\{0,1\}$ defined as follows.
    \[
        f_A(g)=
        \begin{cases}
        1 & g\in A \\
        0 & g\notin A
        \end{cases}
    \]
    The powerset $2^F$ can be given a topology by identifying it with the product topology on $\{0,1\}^{F_n}$.
    Here we use the discrete topology on $\{0,1\}$.

    Given an element $g\in F_n$, define $U_g$ to be the collection of subsets of $F_n$ which contain $g$ so $U_g^c=2^F-U_g$ is the collection of sets which do not contain $g$.
    For each $g$, $U_g$ and $U_g^c$ are open in $2^F$.

    For each $k$, define $\mathcal{S}_k\subset 2^F$ be the collection of all subsets $A\subset F_n$ such that $A\cap W_k$ is a $k$-precone of $F_n$ preserved by $\beta$.
    This is a nested family as follows.
    \[
        \mathcal{S}_1\supset\mathcal{S}_2\supset\mathcal{S}_3\supset\cdots
    \]
    Since $\beta$ preserves a $k$-precone for each $k$, each $\mathcal{S}_k$ is nonempty.
    \bigskip{}

    \noindent\textit{Claim: Each $\mathcal{S}_k$ is a closed subset of $2^F$.}
    \bigskip{}

    Consider a ``point" $S$ in $\mathcal{S}_k^c=2^F-\mathcal{S}_k$; then $S$ is a subset of $F_n$ such that $S\cap W_k$ is not a $k$-precone.
    Since $W_k$ is finite, the subset of $2^F$
    \[
        U_S=\Big(\bigcap_{g\in S\cap W_k} U_g\Big)\cap\Big(\bigcap_{g\notin S^c\cap W_k} U_g^c\Big)
    \]
    is open in $2^F$.
    Note that a set $A$ is in the collection $U$ if and only if $A\cap W_k=S\cap W_k$.
    It follows that $S\in U_S\subset \mathcal{S}_k^c$ for each set $S\in \mathcal{S}_k^c$.
    Thus, $\mathcal{S}_k^c$ is $\bigcup_{S\in \mathcal{S}_k^c} U_S$.
    Since each $U_S$ is open, $\mathcal{S}_k^c$ is open.
    Therefore, $\mathcal{S}_k$ is closed.
    This completes the proof of the claim.

    Since the discrete topology on $\{0,1\}$ is compact, $2^F$ is compact by the Tychonoff theorem \cite{Tych30}.
    Since each $\mathcal{S}_k$ is closed, $\bigcap \mathcal{S}_k$ is also nonempty by Lemma \ref{lem:nestednonempty}.

    Let $P\in \bigcap\mathcal{S}_k$, and let $P_k=P\cap W_k$ for each positive integer $k$.
    Thus, $P=\bigcup P_k$, and when $k\leq l$, $P_k\subset P_l$. Since $P\in\mathcal{S}_k$ for each $k$, each $P_k$ is a $k$-precone preserved by $\beta$, and these cones are nested. 
\end{proof}

With the help of Lemma \ref{lem:nestednonempty} and Lemma \ref{lem:preconesnested}, we are now ready to prove Proposition \ref{prop:whyalgworks}.

\begin{proof}[Proof of Proposition \ref{prop:whyalgworks}]
    As noted in the discussion before Proposition \ref{prop:whyalgworks}, if $\beta$ preserves a positive cone $P$, then $P\cap W_k$ is a $k$-precone preserved by $\beta$.

    Now suppose that for each positive $k$, the braid $\beta$ preserves a $k$-precone. By Lemma \ref{lem:preconesnested} we can assume that these precones are nested. By Lemma \ref{lem:nestednonempty}, the union of these nested precones is a positive cone for $F_n$ preserved by $\beta$.
\end{proof}

\section{Words with zero exponent sum}\label{sec:0-sum}

In order to improve the efficiency of our algorithm, we would like to minimize the number of $k$-precones we need to consider. Towards this end, we focus on words in the free group with zero exponent sum, a subgroup which we call $K_0$. We prove the following general proposition which will imply that certain bi-orders of $K_0$ preserved by $\beta$ are related to bi-orders on the free group preserved by $\beta$.

\begin{prop}\label{prop:inducedcone} Suppose we have the following short exact sequence where $A$ is an abelian group.
\[ 1\rightarrow K\rightarrow F_n\stackrel{\phi}{\rightarrow} A  \rightarrow 1\]

Let $\beta$ be a braid (or more generally any automorphism of $F_n$) with the property $\beta(K)=K$. 
Suppose also that there exists a positive cone $P_A$ of $A$ that is preserved by $\beta_*$ the $\phi$-induced action of $\beta$ on $A$. Then the following statements are equivalent:
\begin{enumerate}
    \item There exists a positive cone $P_K$  of $K$  that is preserved by the action of $\beta$ and invariant under conjugation by elements of $F_n$.
    \item There exists a positive cone $P_F$ of $F_n$ preserved by the action of $\beta$.
\end{enumerate}

\end{prop}

\begin{proof}
$(\Rightarrow)$ Given $P_K$, the cone $P_F$ we seek is 

\[
    P_F:=\{x\in F_n\,|\,\phi(x)\in P_A, \text{ or } \phi(x)=0 \text{ with } x\in P_K\}.
\]
%Then $P_F$ is a conjugate invariant positive cone of $F_n$ that is preserved by the action of $\beta$ when $gP_Kg^{-1}=P_K$ for every $g\in F_n$.

It is a routine proof to show that $P_F$ is a conjugate invariant positive cone of $F_n$ ; see \cite[Proposition 2.1]{PerRolf03} and \cite[Problem 1.23]{CR16}.

To see that $P_F$ is invariant under action of $\beta$,
let $x\in P_F$. If $x\notin K$, then $\phi(x)\in P_A$.
Since $P_A$ is closed under the action of $\beta_*$, we have that $\phi(\beta(x))=\beta_*(\phi(x))\in P_A$ so $\beta(x)\in P_F$.
If $x\in K$, then $x\in P_K$. Since $P_K$ is closed under the action of $\beta$, then  $\beta(x)\in P_K$ and $\beta(x)\in P_F$.

%\textbf{$P_F$ is conjugation invariant}:
%Let $g\in F_n$ and supposed $x\in P_F$.
%If $x\in K$, or $x\in P_K$, then by assumption $gxg^{-1}\in P_K$ (note that $gxg^{-1}\in K$ as kernels are normal).
%If $x\notin K$, then $\phi(x)\neq 0\in P_A$ and since $A$ is abelian $\phi(gxg^{-1})=\phi(x)\in P_A$. 

$(\Leftarrow)$ Suppose that we have a a conjugate invariant positive cone $P_F$ of $F_n$ preserved by $\beta$. Then we define $P_K:=P_F\cap K$. Since $K$ and $P$ are both conjugate invariant in $F_n$, $P_{K}$ is a positive cone of $K$ that is closed under conjugation by $F_n$.
Since $K$ is preserved by $\beta$ by assumption, then $P_K$ is preserved by $\beta$.
\end{proof}

Let $t:F_n\to\Z$ be the exponent sum map, and let $K_0$ be the kernel of $t$. 

\begin{lemma} \label{lem:betakzero}
    For any braid $\beta$, we have that $t(\beta (w))=t(w)$ for all $w\in F_n$.
\end{lemma}

\begin{proof}
    We see from Equation (\ref{eq:artin}) that the action of a generator $\sigma_i$ preserves the exponent sum of a word in $F_n$.
    It follows that for any braid $\beta$, we have that $t(\beta (w))=t(w)$ for $w\in F_n$.
\end{proof}

Taking $K=K_0$ and $\phi$ to be $t$, the following corollary and theorem follow directly from Proposition \ref{prop:inducedcone} and Lemma \ref{lem:betakzero}.

\begin{cor} \label{cor:0-exp}
    A braid $\beta$ preserves a bi-order of $F_n$ if and only if it preserves a bi-order of $K_0$ that is conjugation invariant under elements of $F_n$.
\end{cor}

\begin{proof}
    By Lemma \ref{lem:betakzero}, we have that $t(w)=0$ if and only if $t(\beta(w))=0$ for any word $w\in F_n$.
    Since $K_0=t^{-1}(0)$, we have $\beta(K_0)=K_0$.

    Furthermore, Lemma \ref{lem:betakzero} implies that the $t$-induced action of $\beta$ on $\Z$ is the identity map so $\beta$ preserves a positive cone of $\Z$.
    The corollary now follows from Proposition \ref{prop:inducedcone}.
\end{proof}

\allgenspos*

\begin{proof}
    It suffices to prove the forward direction.
    Suppose $P'$ is a positive cone of $F_n$ preserved by $\beta$, and denote $P_{K_0}:=P'\cap K_0$.
    Then,  $P_{K_0}$ is a positive cone of $K_0$ preserved by $\beta$ and conjugation in $F_n$.
    As in the proof of Proposition \ref{prop:inducedcone}, we can define the positive cone.
    \[
        P:=\{x\in F_n\,|\,t(x)>0, \text{ or } t(x)=0 \text{ with } x\in P_{K_0}\}.
    \]
    By construction, any word in $F_n$ with positive exponent sum is in $P$.
\end{proof}

In theory, this means that we can seed our $k$-precones immediately with all words of positive exponent sum. In practice, our algorithm instead searches for intersections of $k$-precones with $K_0$ which are preserved by $\beta$ and are not only conjugate invariant in $K_0$, but in the larger group $F_n$, as in the following definition. 

\begin{definition}
    A subset $Q$ of $K_0$ is a \emph{$k$-zerocone} of $F_n$ if
    \begin{enumerate}
        \item $(Q\cdot Q)\cap W_k\subset Q$,
        \item $Q\sqcup Q^{-1}\sqcup\{1\}=W_k\cap K_0$, and
        \item $(gQg^{-1})\cap W_k\subset Q$ for all $g\in W_k$.
    \end{enumerate}
\end{definition}

Notice that the conjugation in condition (3) is by all elements of $W_k$, not just the ones in $K_0$.
Now, to relate $k$-zerocones to $k$-precones of $F_n$, we define the set $\mathrm{Pos}_k(Q)$ obtained by adding all positive exponent sum words of length at most $k$ to a $k$-zerocone $Q$ as follows,
    \[
        \mathrm{Pos}_k(Q): = Q\cup \Big(W_k\cap t^{-1}(\Z^+)\Big),
    \]
    where $t:F_n\to\Z$ is the exponent sum map.

%A $k$-precone of $K_0$ is a precone of $F_n$ intersected with $K_0$. to be conjugation invariant under elements of $F_n$

\begin{lemma} \label{lem:zeronconesintersections}
\begin{enumerate}[(a)]
    \item If $P$ is a $k$-precone,
    then the intersection $P\cap K_0$ is a $k$-zerocone.
    \item Suppose $Q$ is a $k$-zerocone. The set $\mathrm{Pos}_k(Q)$ is a $k$-precone.
    \item In particular, the set of $k$-zerocones of $K_0$ are precisely the set of intersections of $k$-precones of $F_n$ and $K_0$.
\end{enumerate}
\end{lemma}

\begin{proof}
 \noindent   (a) Suppose $P$ is a $k$-precone of $F_n$.
    We show that $Q=P\cap K_0$ is a $k$-zerocone.

    \textbf{Condition (1):}
    Suppose $a,b\in Q$.
    Then when $ab\in W_k$, we have that $ab\in P$.
    Thus since $K_0$ is a subgroup of $F_n$, $ab\in P\cap K_0=Q$.

    \textbf{Condition (2):} \begin{align*}
        W_k\cap K_0=&\Big(P \sqcup P^{-1} \sqcup \{1\}\Big)\cap K_0 \\
        =&\Big(P\cap K_0\Big) \sqcup \Big(P^{-1}\cap K_0\Big) \sqcup \{1\}\\
        =&Q\sqcup Q^{-1}\sqcup \{1\}
    \end{align*}    

     \textbf{Condition (3):}
    Suppose $g\in W_k$ and $x\in Q$.
    When $gxg^{-1}\in W_k$, we have that $gxg^{-1}\in P$.
    Thus since $K_0$ is normal, $gxg^{-1}\in P\cap K_0=Q$.\\

 \noindent   (b) For the second statement, we show that $\mathrm{Pos}_k(Q)$ is a $k$-precone.
    
    \textbf{Condition (1):}
    Suppose $a,b\in \mathrm{Pos}_k(Q)$.
    If $a$ and $b$ are both in $Q$ then by definition of a $k$-zerocone, when $ab\in W_k$, we have that $ab\in Q\subset \mathrm{Pos}_k(Q)$.
    If either $a$ or $b$ is not in $Q$, then since $Q\subset K_0$, and $K_0=\ker t$, we have $t(ab)>0$.
    Thus, when $ab\in W_k$, we have that $ab\in W_k\cap t^{-1}(\Z^+)\subset \mathrm{Pos}_k(Q)$.

    \textbf{Condition (2):}
    First, we note that
    \[
        \mathrm{Pos}_k(Q)^{-1} = Q^{-1}\cup \Big(W_k\cap t^{-1}(\Z^-)\Big) .
    \]
    For every element $x$ in $W_k$, exactly one of $t(x)>0$, $t(x)<0$, or $x\in K_0$ is true.
    Thus, we have that
    \begin{align*}
        W_k=&\Big(W_k\cap t^{-1}(\Z^+)\Big)\sqcup \Big(W_k\cap K_0\Big) \sqcup \Big(W_k\cap t^{-1}(\Z^-)\Big) \\
        =&\Big(W_k\cap t^{-1}(\Z^+)\Big)\sqcup Q\sqcup\{1\}\sqcup Q^{-1} \sqcup \Big(W_k\cap t^{-1}(\Z^-)\Big)\\
        =&\mathrm{Pos}_k(Q)\sqcup \{1\} \sqcup \mathrm{Pos}_k(Q)^{-1}
    \end{align*}    
    
     \textbf{Condition (3):}
    Suppose $g\in W_k$ and $x\in \mathrm{Pos}_k(Q)$.
    If $x\in Q$ then when by definition of a $k$-zerocone, when $gxg^{-1}\in W_k$, we have that $gxg^{-1}\in Q\subset \mathrm{Pos}_k(Q)$.
    If $x\in W_k\cap t^{-1}(\Z^+)$ then $t(gxg^{-1})=t(x)>0$.
    Thus, when $gxg^{-1}\in W_k$, we have that $gxg^{-1}\in W_k\cap t^{-1}(\Z^+)\subset \mathrm{Pos}_k(Q)$.\\    

    \noindent (c) For the final statement, we have already shown that an intersection of a $k$-precone and $K_0$ is a $k$-zerocone by Part (a).
    For the other inclusion, we have that any $k$-zerocone $Q$ is the intersection $\mathrm{Pos}_k(Q)\cap K_0$ by definition of  $\mathrm{Pos}_k(Q)$.
\end{proof}

\begin{defn}
We say a $k$-zerocone $Q_k$ of $K_0$ is \emph{preserved by} an automorphism $\varphi$ of $F_n$ if $\varphi(Q_k)\cap W_k\subset Q_k$.
\end{defn}

\begin{lemma} \label{lem:preconestozerocones}
    For each $k\in\N$, the braid $\beta$ preserves a $k$-precone of $F_n$  if and only if $\beta$ preserves a $k$-zerocone of $K_0$.
\end{lemma}

\begin{proof}
    Suppose $\beta$ preserves a $k$-precone $P_k$ of $F_n$.
    Then, $Q_k=P_k\cap K_0$ is a $k$-zerocone by Lemma \ref{lem:zeronconesintersections}.
    Since $\beta$ is a bijection, then $\beta(Q_k)=\beta(P_k\cap K_0)=\beta(P_k)\cap\beta(K_0)=\beta(P_k)\cap K_0$. Using the fact that $\beta$ preserves $P_k$, we have $\beta(P_k)\cap W_k\subset P_k$, and so
    \[
        \beta(Q_k)\cap W_k=W_k\cap \beta(P_k)\cap K_0 \subset P_k \cap K_0=Q_k.
    \]

    Conversely, suppose $\beta$ preserves a $k$-zerocone $Q_k$ of $K_0$ so $\beta(Q_k)\cap W_k\subset Q_k$.
    Define the $k$-precone $P_k$ as in Lemma \ref{lem:zeronconesintersections} as follows.
    \[
        P_k = Q_k\cup \Big(W_k\cap t^{-1}(\Z^+)\Big)
    \]
        
    Since $\beta$ is bijective and doesn't affect exponent sum,
    \begin{align*}
        \beta(P_k)\cap W_k =& \Big[\beta(Q_k)\cup \Big(\beta(W_k)\cap t^{-1}(\Z^+)\Big)\Big]\cap W_k\\
        =& \Big(\beta(Q_k)\cap W_k\Big)\cup \Big(\beta(W_k)\cap t^{-1}(\Z^+)\cap W_k\Big) \\
        \subset & Q_k \cup \Big(t^{-1}(\Z^+)\cap W_k\Big) = P_k.
    \end{align*}    
\end{proof}

Combining Proposition \ref{prop:whyalgworks} and Lemma \ref{lem:preconestozerocones}, we can detect non-order-preserving braids by obstructing $k$-zerocones of $K_0$ by the following proposition.

\begin{prop}
     \label{prop:whyzeroalgworks}
     %If there exists an integer $k$ so that $\beta$ does not preserve any $k$-precone of $K_0$, then $\beta$ is not order-preserving.
     %Let $\mathcal{Q}_k$ be the set of all $k$-zerocones of $K_0$.
     The braid $\beta$ is not order-preserving if and only if for some integer $k$ the braid $\beta$ does not preserve any $k$-zerocone of $K_0$.
\end{prop}

% \begin{figure}[b]
%    \centering
%    \includegraphics[scale=.5]{Screenshot 2023-09-21 at 3.42.40 PM.png}
%    \caption{Output of the implementation of Algorithm \ref{alg:fancy} applied to the braid $\sigma_1\sigma_2^{-3}$ with word length restriction to $k=4$.}
%    \label{fig:example_output}
% \end{figure}

\section{Algorithms} \label{sec:alg}

 Calegari and Dunfield described a theoretical algorithm for obstructing the left-orderability of a group \cite[Section 8]{CD03}.
For a finitely presented  group $G$ with a solution to the word problem, their algorithm produces an obstruction to left-orderability in finite time when $G$ is not left-orderable group.
When $G$ is left-orderable, their algorithm does not halt.
Taking inspiration from Calegari and Dunfield's work, we describe and implement an algorithm to answer the following question.
\begin{question} \label{Quest:PresPreCone}
    Suppose $\beta$ is an $n$-strand braid, and let $k$ be a positive integer.
    Does $\beta$ preserve a $k$-precone of $F_n$ in the sense of Definition \ref{def:preconepreserve}?
\end{question}
By Proposition \ref{prop:whyalgworks}, a braid $\beta$ is order-preserving if and only if the answer to Question \ref{Quest:PresPreCone} is ``yes'' for every positive integer $k$. 
The recursive algorithm PreservePreCone($\beta$,$\{x_1\}$,$k$) defined in Algorithm \ref{alg:basic} returns True or False when the answer to Question \ref{Quest:PresPreCone} is ``yes'' or ``no'' respectively.
Note that since we are working with the free group, we can solve the word problem by greedy reduction.

Recall that $W_k$ is the set of words in $F_n=\langle x_1,\dots,x_n\rangle$ with word length $k$ or less, as defined in Section \ref{sec:precones}.
Given a braid $\beta$, a positive integer $k$, and a subset $P\subset W_k$, Algorithm \ref{alg:basic} attempts to add elements to a set $P$ until it is a $k$-precone preserved by $\beta$.

To satisfy the definition of $k$-precone preserved by $\beta$, we add the following elements to $P$:
\begin{itemize}
    \item products of elements in $P$,
    \item conjugates of elements in $P$ by the free group generators, and
    \item images under $\beta$ and $\beta^{-1}$ of elements in $P$.
\end{itemize}
We will denote this saturation operation by $S_{\beta}(P)$ which is explicitly defined as follows:
\[
    S_{\beta}(P):=P\cup(P\cdot P)\cup \Big(\bigcup_{i=1}^n\{x_iPx_i^{-1}\}\Big)\cup \beta(P)\cup \beta^{-1}(P) .
\]
Since we are restricted to words in $W_k$, we recursively define $P=S_{\beta}(P)\cap W_k$.
The algorithm repeatedly applies this operation until $P$ is the smallest set containing $P$ which is closed, after restricting to elements in $W_k$, under products, conjugation by elements of $F_n$, and the actions of $\beta$ and $\beta^{-1}$.

If the identity is in $P$, the algorithm returns false.
If all non-trivial words of $W_k$ are in $P$ or $P^{-1}$, then the algorithm returns true.
If neither of these cases are satisfied, we recursively apply Algorithm \ref{alg:basic} as follows.
Order the non-trivial elements of $W_k$ by word length, shortest to longest, using some choice of tiebreaker for elements of the same length.
Let $\alpha$ be the shortest non-trivial element in $W_k$ which is not in $P$ or $P^{-1}$.
Algorithm \ref{alg:basic} returns true if applying the algorithm to either $P$ with $\alpha$ add or $P$ with $\alpha^{-1}$ added returns true.
Otherwise, Algorithm \ref{alg:basic} returns false.

\begin{prettyalg}{PreservePreCone($\beta, P, k$)}{alg:basic}

    \begin{algorithmic}
%\Require $k \geq 2$, braid $\beta$
%\Ensure $\beta$ preserves $P$ or no $P$ exists
\While{$S_{\beta}(P)\cap W_k\not\subset P$} 
    \State $P:=S_{\beta}(P)\cap W_k$
\EndWhile
\If{$1\in P$}
    \State $\rt$False
\EndIf
\If{$P\cup P^{-1}\cup\{1\}=W_k$}
    \State $\rt$True
\EndIf
\State $\alpha:=$ shortest word in $W_k-(P\cup P^{-1}\cup\{1\})$
\State $\rt$PreservePreCone($\beta, P\cup\{\alpha\}, k$) \textbf{or} PreservePreCone($\beta, P\cup\{\alpha^{-1}\}, k$)
\end{algorithmic}

\end{prettyalg}

Algorithm \ref{alg:basic} will return true if any of the recursive calls of Algorithm \ref{alg:basic} returns true.
When Algorithm \ref{alg:basic} returns false, there is binary tree of recursive applications of Algorithm \ref{alg:basic}.
The starting node is the application of Algorithm \ref{alg:basic} to the original set $P$.
The leaves of the tree correspond to applications of Algorithm \ref{alg:basic} where a contradiction was found.
This binary tree is a certificate of the non-order-preservingness of $\beta$.

When executing Algorithm \ref{alg:basic}, every non-trivial element of $W_k$ must be placed in $P$ or $P^{-1}$ at least once either by the function $S_{\beta}$ or during the recursive branching step.
Since the number of words in $F$ with length $k$ is $6\cdot 5^{k-1}$, the time complexity of Algorithm \ref{alg:basic} is at least exponential in $k$.
When implemented, this algorithm does not complete in a reasonable time for $k > 6$.

When $\beta$ maps short words to significantly longer words, it is easy to find a $k$-precone $P_k$ of the free group where $\beta(P_k)\cap W_k$ is small.
(For example, the braid $\sigma_1^3\sigma_2^{-3}\sigma_1$ maps $x_2$ to a word of length 21, and you wouldn't see much of the braid action in $W_k$ until $k=21$ or higher.)
In this case, when $\beta$ is non-order-preserving, $k$ must be large for the answer to Question \ref{Quest:PresPreCone} to be ``no".
This means that in practice, Algorithm \ref{alg:basic} is not so practical for obstructing order-preservingness of most braids.

In our implementation, we make several modifications to improve the effectiveness of obstructing order-preservingness.
First, instead of using the action of $\beta$ and $\beta^{-1}$, we use automorphisms $b=\psi_1\circ\beta$ and $b'=\psi_2\circ\beta^{-1}$ where $\psi_1$ and $\psi_2$ are inner automorphisms.
The automorphisms $\psi_1$ and $\psi_2$ are chosen to minimize the longest possible length of the images $b(w)$ and $b'(w)$.
Lemma \ref{lem:inner_conj} describes that composition with an inner automorphism does not change a preserved positive cone. So the choices of $\psi_1$ and $\psi_2$ only help us to find a contradiction sooner (for smaller $k$) by changing the order in which elements are added to a precone.

\begin{lem}\label{lem:inner_conj}
    A positive cone $P$ is preserved by $\beta$ if and only if $P$ is preserved by $\phi\circ\beta$ for $\phi$ an inner automorphism of $F_n$.
\end{lem}

\begin{proof}
    Suppose $P$ is a positive cone.
    Since $P$ is conjugate invariant, any inner automorphism will preserve $P$.
    Thus, for any braid $\beta$ and inner automorphism $\phi$, we have that $\beta(P)=P$ if and only if $\phi(\beta(P))=P$.
\end{proof}

Second, in light of Corollary \ref{cor:0-exp}, we only add words with exponent sum zero to our prospective $k$-precone.
While this doesn't change the time complexity of the algorithm, it significantly reduces the number of words we need to consider.
To do this, instead of seeding our prospective precone with shortest elements in $W_k$, we seed with words in $Z_k$, the subset of words in $W_k$ with zero exponent sum.

Finally, instead of restricting ourselves to working with words at most length $k$, we allow our algorithm to ``remember" words of longer length without using these extra elements in the computation to $S_{\beta}(P)$.
This means for a given $k$ our algorithm will find contradictions for preserving larger precones without having to perform extra computations.

After these modifications we get Algorithm \ref{alg:fancy}.

\begin{prettyalg}{ModPreservePreCone($\beta, P, E, k)$}{alg:fancy}
\begin{algorithmic}
%\Require $k \geq 2$, braid $\beta$
%\Ensure $\beta$ preserves $P$ or no $P$ exists
\While{$S(P)\cap Z_k\not\subset P$}
    \State $P_*:=S_{\beta}(P)\cap Z_k$
    \State $E:=E\cup(S_{\beta}(P)-P_*)$  \Comment{Tracking elements in cone with length greater than $k$}
\EndWhile
\If{$1\in P_*\cup E$}
    \State $\rt$False
\EndIf
\If{$P_*\cup P_*^{-1}\cup\{1\}=Z_k$}
    \State $\rt$True
\EndIf
\State $\alpha:=$ shortest word in $Z_k-(P_*\cup P_*^{-1}\cup\{1\})$
\State $\rt$ModPreservePreCone($\beta,P\cup\{\alpha\}, E, k$) \textbf{or} ModPreservePreCone($\beta, P\cup\{\alpha^{-1}\}, E, k$)
\end{algorithmic}
\end{prettyalg}

By Remark \ref{rem:favorite_elt}, we can always start with the assumption that $x_1<x_2$, which is the same as seeding Algorithm \ref{alg:fancy} with the set $P=\{x_1^{-1}x_2\}$.
If ModPreservePreCone($\beta,\{x_1^{-1}x_2\},\emptyset, k$) returns true,  if the answer to Question \ref{Quest:PresPreCone} is ``yes".
When ModPreservePreCone($\beta,\{x_1^{-1}x_2\},\emptyset, k$) returns false, we can't conclude that answer to Question \ref{Quest:PresPreCone} is ``no" for $k$, but we can conclude that answer to Question \ref{Quest:PresPreCone} is ``no" for some positive integer by Proposition \ref{prop:whyzeroalgworks}.
 More importantly, the same proposition implies that when the algorithm ModPreservePreCone($\beta,\{x_1^{-1}x_2\},\emptyset, k$) returns false, the braid $\beta$ is not order-preserving.

To achieve the final algorithm, described in the introduction as Algorithm \ref{The_alg}, one simply needs to call ModPreservePreCone($\beta,\{x_1^{-1}x_2\},\emptyset,k$) iteratively with increasing values of $k$ until ModPreservePreCone($\beta,\{x_1^{-1}x_2\},\emptyset,k$) returns false. 

\begin{prettyalg}{input braid $\beta$}{alg:incrementk}
\begin{algorithmic}
\State $k=1$
\State FoundPrecone=True
\While{FoundPrecone is True}
    \If{ModPreservePreCone($\beta,\{x_1^{-1}x_2\},\emptyset,k$) is False}
    \State FoundPrecone=False
    
    \Else{ increase $k$ by 1}
    \EndIf
\EndWhile
\State $\rt$ False
\end{algorithmic}
\end{prettyalg}

As can be seen in Algorithm \ref{alg:incrementk}, if ModPreservePreCone($\beta,\{x_1^{-1}x_2\},\emptyset, k$) never returns false, the algorithm will not terminate. 
However, by Propositions \ref{prop:whyalgworks} and \ref{prop:whyzeroalgworks}, if $\beta$ is not-order-preserving, then for some $k$, ModPreservePreCone($\beta,\{x_1^{-1}x_2\},\emptyset,k$) will return false and the program will terminate in finite time.

\subsection{Implementation}\label{sec:Ex_Outs}
 Algorithms \ref{alg:basic}, \ref{alg:fancy} and \ref{alg:incrementk} are implemented in SageMath and Python. The code for these algorithms is available on Github \cite{JST23}.
 As a practical note, in a work occurring somewhat concurrently to this article, Cai-Clay-Rolfsen developed a sufficient condition to identify order-preserving braids \cite{CCR-Ordered-bases}. 
While Algorithm \ref{alg:incrementk} does not take this condition into account, we implemented this condition as a separate function. So in practice, one should first check if the braid in question satisfies their condition before indefinitely running Algorithm \ref{alg:incrementk}.

 When the implementation of Algorithm \ref{alg:fancy} returns false, it also returns a proof that the braid is not order-preserving. 
 The proof output is a binary tree recording attempts to build precones; see Figure \ref{fig:binaryTree} for a visual representation of this tree for $\s_{1}\s_{2}^{-3}$. The first node is the seeding element of the prospective precone. Each parent node has two child nodes: the first corresponding to an attempt to add a new element $\alpha$ to the precone and the second to adding $\alpha^{-1}$. If the attempt was successful, there will be no proof information and the branching process will continue. If the attempt was unsuccessful, the proof info will output two elements of the attempted precone that are inverses, as well as instructions for how the elements were added to the precone.
 The algorithm proved that $\sigma_1\sigma_2^{-3}$ is not order-preserving, which is a new result.

\begin{prop}
     The braid $\sigma_1\sigma_2^{-3}$ is not order-preserving.
\end{prop}
\begin{proof}

Figure \ref{fig:binaryTree} is a visual representation of the following argument. As stated in Remark \ref{rem:favorite_elt}, we may assume $x_1^{-1}x_2$ is an element of $P$ (this is the seeding element). Let $\alpha=x_2^{-1}x_3$. Either $\alpha$ or $\alpha^{-1}$ is in $P$.
If $\alpha\in P$ (one node element) then the element $x_3^{-1}x_2^{-1}x_1x_2$ and its inverse (contradiction elements) are both in $P$, which is a contradiction.
If $\alpha^{-1}=(x_2^{-1}x_3)^{-1}\in P$ (the other node element) then the element $x_2^{-2}x_1x_2x_3x_2^{-1}$ and its inverse are both in $P$, which is a contradiction.
Thus no such $P$ can exists and $\sigma_1\sigma_2^{-3}$ is not order-preserving.
\end{proof}

\begin{figure}
\begin{center}
\begin{tikzpicture}[node distance=5cm]

\node (start) [rectangle, text width=4cm, draw=black, align=center]{Prospective positive cone seeded with $y=x_1^{-1}x_2$};

\node (case1) [rectangle, below left of=start, text width=5cm, draw=black, align=center]{Contradiction for \\ $\mu=x_3^{-1}x_2^{-1}x_1x_2$: Both $\mu$ and $\mu^{-1}$ in prospective cone.\\ ~ \\ $\mu=\beta (\alpha^{x_3^{-1}}) $ \\
$\mu^{-1}=y^{x_2^{-1}} \alpha$};

\node (case2) [rectangle, text width=5cm, below right of=start, draw=black, align=center]{Contradiction for $\nu=x_2^{-2}x_1x_2x_3x_2^{-1}$: Both $\nu$ and $\nu^{-1}$ in prospective cone.\\ ~ \\ $\nu=\beta((y^{x_3}  \beta(\alpha^{-1}))^{x_2})^{x_2} $ \\
$\nu^{-1}=(\alpha^{-1})^{x_2} y^{x_2^{-1}}$};

\draw (start) -- (case1) node[midway, left, text width = 3.5cm, align = center]{Adding $\alpha=x_2^{-1}x_3$ to cone};
\draw (start) -- (case2) node[midway, right, text width = 3.9cm, align = center]{Adding $\alpha^{-1}=(x_2^{-1}x_3)^{-1}$ \\ to cone};
\end{tikzpicture}
\end{center}
\caption{This tree depicts how Algorithm 30 obstructs order-preservingness of $\sigma_1\sigma_2^{-3}$. For conjugation, we use the notation $g^h:=hgh^{-1}$ for two elements $g,h\in F_n$.}\label{fig:binaryTree}
\end{figure}
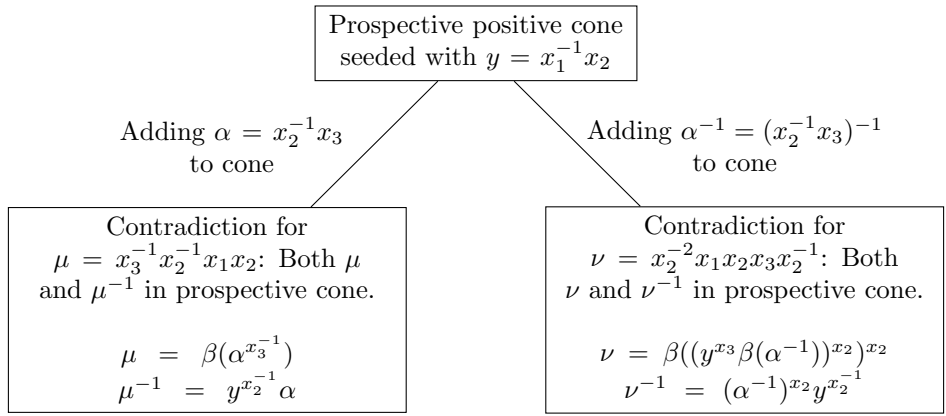     

\section{A Family of Non-Order-Preserving Braids}\label{sec:Family}

We prove that the braids $\sigma_1\sigma_2^{2m+1}$ are not order-preserving.
Our proof is inspired by the computer generated proof resulting from the implemented Algorithm \ref{alg:fancy} applied to the braid $\sigma_1\sigma_2^{-3}$, as discussed in Section \ref{sec:Ex_Outs} and visualized in Figure \ref{fig:binaryTree}.

\braidfam*

\begin{proof}
    
    %Let $l$ be an integer, and consider $\beta=\sigma_1\sigma_2^{-k}$ where $k=2l+1$.

    Let $\psi\in\inn(F_3)$ be conjugation by $w^{-m}$ where $w=x_2^{-1}x_1x_2x_3$,
    let $f$ be the automorphism $\psi\circ\beta$ in $\aut(F_3)$ which is defined by the following action.
    \[
        x_1\longmapsto w^{-m} x_2 w^m \hspace{0.5cm} x_2 \longmapsto x_2^{-1}x_1x_2x_3x_2^{-1}x_1^{-1}x_2
    \]
    \[
        x_3 \longmapsto x_2^{-1}x_1x_2
    \]
    The automorphism $f$ and the braid $\beta$ preserve the same bi-orders of $F_3$ by Lemma \ref{lem:inner_conj}. 

    Suppose $P$ is a positive cone of $F_3$ preserved by $f$.
    We may assume without loss of generality that $x_1^{-1}x_2\in P$.

    %Suppose that $x_2^{-1}x_3\in P$.
    Now, either $x_2^{-1}x_3$ or $x_3^{-1}x_2$ must be in $P$. Suppose first that $x_2^{-1}x_3\in P$. Then 
    \[
        f(x_2^{-1}x_3)=x_2^{-1}x_1x_2x_3^{-1}\in P .
    \]
    Additionally, we have that $x_3x_2^{-1}$ is in $P$ by conjugating $x_2^{-1}x_3$.  However, since $x_1^{-1}x_2$ is also in $P$, we have that $x_1^{-1}x_2f(x_2^{-1}x_3)x_3x_2^{-1}=1\in P$ which is a contradiction.
    %the elements  $x_2^{-1}x_1$ and $x_2x_3^{-1}$ are in $P^{-1}$.
    %Thus, $x_2^{-1}x_1x_2x_3^{-1}\in P^{-1}$,    which is a contradiction.
    
    On the other hand, suppose that that $x_3^{-1}x_2\in P$.
    Then, since $x_1^{-1}x_2\in P$,
    %conj([b(conj([] by x2))] by x2)
    \[
        x_3(x_1^{-1}x_2)x_3^{-1}\cdot f(x_3^{-1}x_2)=x_3x_1^{-2}x_2  \in P .
    \]
    Since,
    \begin{align*}
        f(x_3x_1^{-2}x_2)= & x_2^{-1}x_1x_2 w^{-m} x_2^{-2} w^m x_2^{-1}x_1x_2x_3x_2^{-1}x_1^{-1}x_2 \\
        = & x_2^{-1}x_1x_2 w^{-m} x_2^{-2} w^{m+1} x_2^{-1}x_1^{-1}x_2 ,
    \end{align*}
    we have that
    \[
        x_2^{-2} w \in P
    \]
    after conjugating by $x_2^{-1}x_1x_2 w^{-m}$.
        
    However, since $x_1^{-1}x_2$ and $x_3^{-1}x_2$ are in $P$, 
    the elements  $x_2^{-1}x_1$ and $x_3x_2^{-1}$ are in $P^{-1}$.
    Thus,

    \[
        x_2^{-1}\Big[x_2^{-1}(x_2^{-1}x_1)x_2\cdot x_3x_2^{-1}\Big]x_2= x_2^{-2} w \in P^{-1}
    \]
    which is a contradiction.
\end{proof}

\printbibliography

@book{CR16,
 ISBN = {9781470431068},
 author = {Adam Clay and Dale Rolfsen},
 publisher = {American Mathematical Society},
 title = {Ordered Groups and Topology (GSM-176)},
 address = {Providence, Rhode Island},
 year = {2016}
}

@article {CD03,
    AUTHOR = {Calegari, Danny and Dunfield, Nathan M.},
     TITLE = {Laminations and groups of homeomorphisms of the circle},
   JOURNAL = {Invent. Math.},
  FJOURNAL = {Inventiones Mathematicae},
    VOLUME = {152},
      YEAR = {2003},
    NUMBER = {1},
     PAGES = {149--204},
      ISSN = {0020-9910,1432-1297},
   MRCLASS = {57M50 (37C85 57N10 57R30 57S05)},
  MRNUMBER = {1965363},
MRREVIEWER = {Charles\ Ira\ Delman},
       DOI = {10.1007/s00222-002-0271-6},
       URL = {https://doi.org/10.1007/s00222-002-0271-6},
}

@article{ND20,
author = {Nathan M Dunfield},
title = {{Floer homology, group orderability, and taut foliations of hyperbolic $3$–manifolds}},
volume = {24},
journal = {Geometry \& Topology},
number = {4},
publisher = {MSP},
pages = {2075 -- 2125},
year = {2020},
doi = {10.2140/gt.2020.24.2075},
URL = {https://doi.org/10.2140/gt.2020.24.2075}
}

@article{JJ2,
   title={Residual torsion-free nilpotence, bi-orderability, and two-bridge links}, volume={76}, DOI={10.4153/S0008414X2300007X}, number={2}, journal={Canadian Journal of Mathematics}, author={Johnson, Jonathan}, year={2024}, pages={394–457}
}

@article {JJ1,
      title={Residual torsion-free nilpotence, biorderability and pretzel knots}, 
      author={Jonathan Johnson},
      ljournal={Algebraic and Geometric Topology},
      journal={Algebr. Geom. Topol.},
      year={2023},
      volume={23},
      number={4},
      pages = {1787-1830}

}

@software{JST23,
author = {Johnson, Jonathan and Scherich, Nancy and Turner, Hannah},
license = {GPL-3.0},
month = oct,
title = {{order-preserving-braids}},
url = {https://github.com/jjohnson524/order_preserving_braids},
version = {1.0},
year = {2023}
}

@misc{CCR-Ordered-bases,
      title={Ordered bases, order-preserving automorphisms and bi-orderable link groups}, 
      author={Tommy Wuxing Cai and Adam Clay and Dale Rolfsen},
      year={2024},
      eprint={2406.18876},
      archivePrefix={arXiv},
      primaryClass={math.GR},
      url={https://arxiv.org/abs/2406.18876}, 
}

@article {KR-BraidsOrderings,
    AUTHOR = {Kin, Eiko and Rolfsen, Dale},
     TITLE = {Braids, orderings, and minimal volume cusped hyperbolic
              3-manifolds},
   JOURNAL = {Groups Geom. Dyn.},
  FJOURNAL = {Groups, Geometry, and Dynamics},
    VOLUME = {12},
      YEAR = {2018},
    NUMBER = {3},
     PAGES = {961--1004},
      ISSN = {1661-7207},
   MRCLASS = {57M07 (20F60 57M27)},
  MRNUMBER = {3845714},
MRREVIEWER = {Thomas Koberda},
       DOI = {10.4171/GGD/460},
       URL = {https://doi.org/10.4171/GGD/460},
}

@book {Mur-3braids,
    AUTHOR = {Murasugi, Kunio},
     TITLE = {On closed {$3$}-braids},
      NOTE = {Memoirs of the American Mathmatical Society, No. 151},
 PUBLISHER = {American Mathematical Society, Providence, R.I.},
      YEAR = {1974},
     PAGES = {vi+114},
   MRCLASS = {55A25},
  MRNUMBER = {0356023},
MRREVIEWER = {J. S. Birman},
}

@book{Munkres2000,
  title={Topology},
  author={Munkres, J.R.},
  isbn={9780131816299},
  lccn={99052942},
  series={Featured Titles for Topology},
  year={2000},
  publisher={Prentice Hall, Incorporated}
}

@article{Ito17,
  title = {Alexander polynomial obstruction of bi-orderability for rationally homologically fibered knot groups},
  author = {Tetsuya Ito},
  year = {2017},
  journal = {New York J. Math.},
  pages = {497--503},
  volume = {23},
}

@article{CDN16,
  title={Testing Bi-orderability of Knot Groups},
  volume={59},
  DOI={10.4153/CMB-2016-023-6},
  number={3},
  journal={Canadian Mathematical Bulletin},
  publisher={Cambridge University Press},
  author={Adam Clay and Colin Desmarais and Patrick Naylor},
  year={2016},
  pages={472–482}
}

@article{Yam17,
  author = {Takafumi Yamada},
  title = {A family of bi-orderable non-fibered 2-bridge knot groups},
  journal = {Journal of Knot Theory and Its Ramifications},
  volume = {26},
  number = {04},
  pages = {1750011},
  year = {2017},
  doi = {10.1142/S0218216517500110},
  URL = {https://doi.org/10.1142/S0218216517500110}
}

@article{Tych30,
  title={\"Uber die topologische Erweiterung von R\"aumen},
  volume={102},
  DOI={10.1007/BF01782364},
  number={1},
  journal={Mathematische Annalen},
  author={Andrey Tychonoff},
  year={1930},
  pages={544–561}
}

@article{PerRolf03,
  title={On orderability of fibred knot groups},
  volume={135},
  DOI={10.1017/S0305004103006674},
  number={1},
  journal={Mathematical Proceedings of the Cambridge Philosophical Society},
  publisher={Cambridge University Press},
  author={Bernard Perron and Dale Rolfsen},
  year={2003},
  pages={147–153}
}
\end{document}